\documentclass[11pt]{article}
\usepackage{amsmath,amssymb,amsthm}
\usepackage[a4paper,left=3cm,right=2cm,top=3cm,bottom=2cm]{geometry}
\usepackage{graphicx}
\usepackage{epsfig}
\usepackage{wrapfig}
\usepackage{float}
\usepackage{subfigure}
\usepackage[english]{babel}
\usepackage{amssymb,amsmath,amsfonts}
\usepackage{graphicx}
\usepackage{xcolor}
\def\bX{\mathbf{X}}
\def\bN{\mathbf{N}}

\def\x{\mathbf{x}}
\def\n{\mathbf{n}}
\def\r{\mathbf{r}}
\def\cH{\mathcal{H}}

\def\cS{\mathcal{S}}
\def\cC{\mathcal{C}}

\def\cG{\mathcal{G}}

\def\cN{\mathcal{N}}

\def\cX{\mathcal{X}}

\def\cZ{\mathcal{Z}}
\def\r{\mathbf{r}}
\def\N{\mathbb{N}}

\def\L{\mathbb{L}}

\def\R{\mathbb{R}}

\def\C{\mathbb{C}}
\def\D{\mathbb{D}}

\def\H{\mathbb{H}}
\def\S{\mathbb{S}}

 \newtheorem{defi}{Definition}[section]
  \newtheorem{teo}[defi]{Theorem}
 \newtheorem{pro}[defi]{Proposition}
 \newtheorem{cor}[defi]{Corollary}
 \newtheorem{lem}[defi]{Lemma}
 
 \newtheorem{remark}{Remark}

\numberwithin{equation}{section}

\begin{document}
\footnotetext{Research partially supported by Ministerio de Educaci\'on Grants No: MTM2013-43970-P,  No: PHB2010-0109, Junta de Anadaluc\'\i a Grants No. FQM325, N0. P06-FQM-01642. Minist\'erio de Ci\^encia e Tecnologia, CNPq Proc. No. 303774/2009-6. Minist\'erio de Educa\c{c}\~ao, CAPES/DGU Proc. No. 23038010833/2010-37. }
\title{A connection between flat fronts in hyperbolic space and minimal surfaces in euclidean space}
\author{ Antonio Mart\'{\i}nez, Pedro Roitman, Keti Tenenblat}
\date{}
\maketitle
{\small 
\noindent Departamento de Geometr\'\i a y Topolog\'\i a, Universidad de Granada, E-18071 Granada, Spain\\ 
e-mail: amartine@ugr.es

\vspace{.05in}

\noindent Departamento de Matem\'atica, Universidade de Bras\'\i lia, 70910-900 Bras\'\i lia, Brasil \\ 
e-mail: roitman@mat.unb.br

\vspace{.05in}
\noindent Departamento de Matem\'atica, Universidade de Bras\'\i lia, 70910-900 Bras\'\i lia, Brasil \\ 
e-mail: K.Tenenblat@mat.unb.br}
 \begin{abstract}
A geometric construction is provided that associates to a given flat
front in $\H^3$ a pair of minimal surfaces in $\R^3$ which are related
by a Ribaucour transformation.  This construction is generalized
associating to a given frontal in $\H^3$, a pair of frontals in $\R^3$
that are envelopes of a smooth congruence of spheres. The theory of
Ribaucour transformations for minimal surfaces is reformulated in
terms  of a complex Riccati ordinary differential equation for a
holomorphic function. This enables one to simplify and extend the
classical theory, that in principle only works for umbilic free and
simply connected  surfaces,  to surfaces with umbilic points and non
trivial topology. Explicit examples are included.
 \end{abstract}
\noindent \hspace*{2ex} 2000 {\it  Mathematics Subject Classification}: 53A35, 53C42

\noindent \hspace*{2ex}{\it Keywords:} Minimal surfaces, Flat fronts, Ribaucour transformations,
hyperbolic space.
\section{Introduction}
It is well known that minimal surfaces in Euclidean space $\mathbb{R}^3$ and flat fronts in hyperbolic space $\mathbb{H}^3$ admit a holomorphic representation. They also share in common the fact that there are many interesting global theorems about their geometry and topology,  see for instance \cite{GMM}, \cite{KUY} and \cite{R}. From the point of view of partial differential equations, both classes are intimately related to the Monge-Amp\`{e}re equation.
$$
\det{(\nabla^2\varphi)}=1.
$$
However, despite these similarities, as far as we know, there is no direct geometric link between these two classes of surfaces that are immersed in different ambient spaces.

What we offer in this work is a geometric construction that associates to a
given flat front  in $\mathbb{H}^3$ a pair of minimal surfaces in $\mathbb{R}^3$ that are related by a Ribaucour transformation. 

This construction is a particular case of a geometric method to associate a given frontal   in $\mathbb{H}^3$ to a pair of frontals in $\mathbb{R}^3$ that are the envelopes of a smooth congruence of spheres. We believe that this construction will help to unravel interesting relation between surfaces immersed in $\mathbb{H}^3$ and $\mathbb{R}^3$.

Ribaucour transformations for minimal surfaces were studied by \cite{Bi} and revisited in \cite{CFT}, \cite{CFT2}, and they can be viewed as a method to generate new examples of minimal surfaces by starting with a given simpler one. In this transformation process lines of curvature are preserved and the transformed surface might have new planar ends.

As an application, we discuss in detail how the classical theory of Ribaucour transformations of minimal surfaces can be reformulated in terms of a complex Riccati ordinary differential equation for a holomorphic function.  We will show how this fact enables one to simplify and extend the classical theory, that only works in principle for umbilic free and simply connected surfaces, to surfaces with umbilic points and non trivial topology.

We also show how the classical results about the Riccati equation can be used to control the asymptotic behavior of the ends of the transformed surface.

Finally, as an example, we also compute some Ribaucour transformations of the trinoid of Jorge-Meeks and compute the flat front associated to a Ribacour transformation of the catenoid. 
 
The geometric connection between minimal and flat  fronts  mentioned above stems from the classical idea of viewing the Lorentz-Minkowski 4-space $\mathbb{L}^4$ as the space of oriented spheres (including points as a limiting case) in euclidean space. In this way, to every surface immersed in $\mathbb{H}^3 \subset \mathbb{L}^4$ we can associate a smooth two parameter family of spheres in $\mathbb{R}^3$, or, in other words, a congruence of spheres.

Roughly speaking, we will show that if we start with a flat front $S \subset \mathbb{H}^3$ and transform the associated congruence of spheres in a convenient way, we end up with another congruence of spheres that has a pair of minimal surfaces as envelopes of the congruence. The process can also be reversed in the sense that if we start with a pair of minimal surfaces that are the envelopes of a congruence of spheres, and such that the lines of curvature correspond pointwise between the envelopes, then this congruence of spheres can be itself transformed in such a way that it corresponds to a flat front in $\mathbb{H}^3$. 

This work is organized as follows. In section \ref{gc} we briefly review Ribaucour transformations and  discuss the geometric construction that associates  a  frontal in $\mathbb{H}^3$ to a pair of frontals in $\mathbb{R}^3$. As a consequence of this construction,  we also obtain a Small's representation formula for frontals in $\mathbb{H}^3$.

In section \ref{flatminimal} we treat the special and important case where the frontal in $\mathbb{H}^3$ is a flat front. In this case, we prove the pair of surfaces in $\mathbb{R}^3$ constitute a Ribaucour pair of minimal surfaces. Section \ref{examples} is devoted to new examples of Ribaucour pairs that lie outside the scope of the classical theory of Ribaucour transformations. Finally, in section \ref{finalremarks} we collect our concluding remarks.
   
\section{ A geometric construction}
\label{gc}
This section is devoted to an explanation of a geometric method which  gives a canonical relationship between a flat front in $\mathbb{H}^3$ and a pair of minimal surfaces in $\mathbb{R}^3$. 

\subsection{The space of oriented spheres as a bridge between $\H^3$ and $\R^3$}
\label{spacespheres}
 
We are going to consider frontals which are, up to some degenerate cases,   surfaces admitting singularities but with a globally defined Gauss map that can be smoothly extended across the singular set. Actually, we  will explain how to relate a given frontal in the hyperbolic space to a pair  of  frontals in $\R^{3}$.

\

Before proceeding to our description, let us clarify some notations and terminology about frontals and fronts, see  \cite{SUY, SUY2, SU} for more details. 
\\
Let $\Sigma$ be an oriented  2-manifold and  $(M^3,g)$ an oriented Riemannian 3-manifold with unit tangent bundle $T_1M^3$. A smooth map  $\cX : \Sigma \longrightarrow M^3$,  is a 
{\sl frontal} if there exists a smooth unit  normal vector field $\cN $ of $M^3$ along $\cX$, that is, 
\begin{equation}
g(d\cX(X),\cN) = 0, \qquad  \forall \ X\in T\Sigma. \label{leg}
\end{equation}
In addition, if $f=(\cX ,\cN):\Sigma \longrightarrow T_1M^3$ is an immersion,  $\cX$ is called a {\sl front}.
\\
The vector field $\cN$ is called the {\sl unit normal} of  $\cX$. The first, second and third  fundamental forms are defined in the same way as for surfaces. 
\\
A point $p\in \Sigma$ is {\sl a singular point}  if $\cX$  is not an immersion at $p$. The set ${\cS_\cX}$ of singular points of $\cX$ is called the {\sl singular set} of $\cX$.
\begin{remark}\label{rim} {\rm It is remarkable  that every minimal front $\cX:\Sigma \longrightarrow \R^3$ is always a minimal immersion, that is $\cS_\cX=\emptyset$.}
\end{remark}

Let $\L^{4}$ be the Lorentz-Minkowski $4$-space endowed with the usual Lorentz metric, $\ll.,.\gg$ and $\mathbb{H}^{3} = \{\r \in \L^{4}\ | \ll \r,\r\gg = -1\}$ be the hyperboloid model for the hyperbolic space of constant sectional curvature $-1$. 

To understand the geometry of surfaces in $\mathbb{H}^3$ it is essential to consider the positive null cone $\N^3$, defined by
$$ \N^3=\{\r \in \L^{4}\ | \ll \r,\r\gg = 0\}.$$
If one considers for all $\r \in \N^3$ the half line $[\r]$ spanned by $\r$, then the ideal boundary of $\H^3$ can be regarded as the quotient of $\N^3$ under this action. In addition, the induced metric is well-defined up to a factor and the ideal boundary of $\H^3$ inherits a natural conformal structure as the quotient $\N^3/\R^+$. Thus, we will from now on identify the ideal boundary of $\H^3$ with the unit sphere $\mathbb{S}^2$.

We will now consider $\L^{4}$ as the space of oriented spheres in $\R^{3}$ in the classical way, see \cite{BL}, that is, we associate to the sphere centered at $\x$ with radius $r>0$ and orientation given by the inner normal the point $(r,\x)\in \L^{4}$. To the sphere with opposite orientation we associate the point $(-r,\x)$. Points of $\R^{3}$ are treated as a limiting case where the spheres have radius zero.

The key idea to pass from a frontal  in $\mathbb{H}^3\subset \L^{4}$ to a pair of frontals in $\mathbb{R}^3$ is to consider the smooth  two parameter family of spheres defined by the starting frontal (called a congruence of spheres) and its two envelopes. 

We recall that a surface $\Sigma$ is an envelop of a congruence of spheres if $\forall p\in \Sigma$ there is one sphere $S_{p}$ of the congruence that is tangent to it at $p$. Since in this work we deal with  frontals, we will adopt the definition that a  frontal   $\Sigma$ is an envelop of a congruence of spheres if $\forall p\in \Sigma$ there is one sphere $S_{p}$ of the congruence such that the normal vector to $\Sigma$ and $S_{p}$ at $p$ are parallel.

\

We start our considerations by establishing how the geometry of a  frontal  in $\mathbb{H}^3$ is related to two frontals in $\R^3$ that are  envelopes of the induced congruence of spheres.

 Let $\Sigma$ be an oriented  2-manifold and $\bX:  \Sigma \longrightarrow \H^3$,
$ \bX:= (r, \x)$ a frontal  with unit normal  $\bN=(s,\n)$ and hyperbolic Gauss maps $\cG_+$ and $\cG_-$. Then $r^2\neq s^2$ and,  if we denote by $\Pi:\S^2\longrightarrow \C\cup\{\infty\}$ the usual stereographic projection,   we can write, 
\begin{align} \cG_+ &= \Pi\circ \cN^+, \qquad  \bX+\bN =  (r+s)(1,\cN^+)\\ 
 \cG_- &= \Pi\circ \cN^-, \qquad  \bX-\bN =  (r-s)(1,\cN^-).
\end{align}
\begin{pro}\label{p1}
$\bX$  is determined  by a unique pair of frontals $\cX^+,\cX^-:\Sigma \longrightarrow \R^3$ satisfying 
\begin{itemize}
\item $\cN_+$ and $\cN_-$ are the unit normals of $\cX^+$ and $\cX^-$, respectively.
\item $\cX^+$ and $\cX^-$ are envelopes of a congruence of spheres and the following symmetry condition hold: 
\begin{equation}
 \|\cX^+\|^2 \ \cX^- + \cX^+ =  \|\cX^-\|^2  \ \cX^+ + \cX^- = 0,\label{sc}
\end{equation}
where  $\|.\|$ denotes the usual Euclidean norm.
\end{itemize}
\end{pro}
\begin{proof}
If we consider the smooth maps $\cX^+,\cX^-:\Sigma \longrightarrow \R^3$ given by
  \begin{align}
\cX^+ &= \x - \frac{r}{r+s}(\x+\n), \label{a1}\\
\cX^- &= \x - \frac{r}{r-s}(\x-\n),\label{a2}
\end{align}
it is fairly easy to check that  $\cN^+$ and $\cN^-$, given by
\begin{align}
\cN^+&=\frac{1}{r+s}(\x+\n),\label{n1}\\
\cN^-&=\frac{1}{r-s}(\x-\n).\label{n2}
\end{align}
satisfy
\begin{equation}
\|\cN^+\| = \|\cN^-\| = 1, \quad < \cN^+,d\cX^+> = < \cN^-,d\cX^->=0, \label{n3}
\end{equation}
where by $< . \ ,  . >$ we will denote the standard inner product in $\R^3$. 

From \eqref{n3},   $\cN^+$ ( resp. $\cN^-$) is  the unit normal  of $\cX^+$  (resp.  $\cX^-$)  and from \eqref{a1}, \eqref{a2}, \eqref{n1} and \eqref{n2} we have that \eqref{sc} holds and  $\cX^+, \cX^-$ are the envelopes of a congruence of spheres with center
$$\x =  \cX^+ + r \  \cN^+ = \cX^-  + r \ \cN^-. $$
and radius $r$.

\

Conversely, assume $\cX^+, \cX^-:\Sigma \longrightarrow \R^3$ satisfy (\ref{sc}), if they are the envelopes of a congruence of spheres with center
$$ \x= \cX^+ + r \  \cN^+ = \cX^-  + r \ \cN^-$$ and radius $r$ such that  $\cN^+\neq \cN^-$ everywhere. 
Then,  from \eqref{sc}, we can write 
\begin{align}
& \cX^+= \lambda  (\cN^+ - \cN^-), \quad \cX^- = - \mu (\cN^+ - \cN^-) , \label{eta1}\\
& 2\lambda \mu (1- <\cN^+,\cN^->)= 1,\label{eta2}
\end{align}
for some smooth functions $\lambda$ and $\mu$.

 From the   
above expressions,  if we consider  $\rho^+$ and $\rho^-$ the support functions of $\cX^+$ and $\cX^-$, respectively, and take differentiation in \eqref{eta1}, then the following expressions hold
\begin{align}
\rho^+ &=<\cX^+,\cN^+>= \lambda (1 - <\cN^-,\cN^+>) = \frac{1}{2\mu},\label{rhop}\\
\rho^- &=<\cX^-,\cN^->= \mu (1 - <\cN^-,\cN^+>) = \frac{1}{2\lambda} ,\label{rhom}\\
\frac{d\lambda}{\lambda}  & = \frac{<d\cN^-, N^+> }{ 1 - < \cN^+,\cN^->}, \qquad \frac{d\mu}{\mu} = \frac{<d\cN^+, N^-> }{ 1 - < \cN^+,\cN^->}, \label{dnu}
\end{align}
and we can  recover 
the frontal   $\bX: \Sigma \longrightarrow \H^3$   given by
 \begin{equation}
\bX  =  -\frac{1}{2\rho^+} (1,\cN^+ ) -  \frac{1}{2\rho^-} (1,\cN^- ) \label{eq5}
\end{equation}
whose unit normal vector $\bN$ can be written as
 \begin{equation}
\bN  = -\frac{1}{2\rho^+} (1,\cN^+ ) +  \frac{1}{2\rho^-} (1,\cN^- ). \label{eq6}
\end{equation}
Moreover, from (\ref{eq5}) and (\ref{eq6}), 
\begin{equation}\label{eqnueva}
\bX + \bN = -\frac{1}{\rho^+} (1, \cN^+), \qquad \bX - \bN = -\frac{1}{\rho^-} (1, \cN^-).
\end{equation}
Thus, the hyperbolic Gauss maps  $\cG_+$  and $\cG_-$ of $\bX$ are  $\Pi\circ\cN^+$ and $\Pi\circ \cN^-$, respectively.
\end{proof}

In the above notations we have 
\begin{lem}\label{lemafronts}
$\bX$ is a front if and only if $\cX^+$ and $\cX^-$ are fronts.
\end{lem}
\begin{proof}
From \eqref{a1}, \eqref{a2}, \eqref{n1},  \eqref{n2} and by a straightforward computation,   the matrix $(d\cX^\pm, d\cN^\pm)$ has the same rank as  the matrix $(d\x - dr \ \cN^\pm, d\n - ds \ \cN^\pm)$. Thus, the proof follows having in mind that
\begin{align*}
& < d\x - dr \cN^\pm, d\x - dr \cN^\pm> = \ll d\bX, d\bX\gg, \\
& < d\n - ds \cN^\pm, d\n - ds \cN^\pm> = \ll d\bN, d\bN\gg.
\end{align*}
\end{proof}
\begin{defi}
{\rm We say that $\cX^+$ and $\cX^-$ are the {\sl associated} frontals (or fronts) of $\bX$.}
\end{defi}
Because the functions $\rho^+$ and $\rho^-$ are determined by  \eqref{rhop},  \eqref{rhom}  and \eqref{dnu} in terms of $\cN^+ $ and $\cN^-$,   equation \eqref{eq5}  gives a  representation formula  of any frontal in $\H^3$ in terms of their hyperbolic Gauss maps, $\cG_+= \Pi\circ\cN^+$ and $\cG_-= \Pi\circ\cN^-$. 
\\
In fact, we can prove the following  Small's representation formula for frontals in $\H^3$:
\begin{teo}\label{t0} Let $\cG_+, \cG_-: \Sigma \longrightarrow \C\cup\{\infty\}$ be  two smooth complex functions such that $\cG_+\neq \cG_-$ everywhere. Then, there exists a frontal $\bX:\Sigma \longrightarrow\H^3$  with hyperbolic Gauss maps $\cG_+$ and $\cG_-$ if and only if 
\begin{equation}
\Re \left( \int_\gamma \frac{d \cG_+}{\cG_+ - \cG_-}\right) = 0, \qquad \text{ for any loop $\gamma$ in $\Sigma$}, \label{existence}
\end{equation}
where by $\Re$ we denote the real part. Moreover, in this case we can write $\bX$ and its unit normal $\bN$ as 
 \begin{align}
\bX  &=  -\frac{1}{2\rho^+} (1,\Pi^{-1}\circ \cG_+ ) -  \frac{1}{2\rho^-} (1,\Pi^{-1}\circ \cG_- ), \label{eq5g}\\
\bN & = -\frac{1}{2\rho^+} (1,\Pi^{-1}\circ \cG_+  ) +  \frac{1}{2\rho^-} (1,\Pi^{-1}\circ \cG_-), \label{eq6g}
\end{align} where 
\begin{align} \rho^+ &= \frac{\|\xi_+\|^2}{1 + \|\cG_+\|^2}, \qquad \xi_+= c_0 \exp (\int \frac{d\cG_+}{\cG_+-\cG_-}),\label{rhoplus}\\
 \rho^- &= \frac{\|\xi_-\|^2}{1 + \|\cG_-\|^2}, \qquad \xi_-= c_1 \exp (\int \frac{d\cG_-}{\cG_--\cG_+}),\label{rhominus}
 \end{align}
 $c_0$ and $c_1$ are non zero complex numbers such that $ \rho^- \rho^+  = \cG_+ - \cG_-$.
\end{teo}
\begin{proof}
It is clear from \eqref{rhop}, \eqref{rhom},  \eqref{dnu} and \eqref{eq5} that there exists $\bX$ if and only if  there exist $\lambda$ and $\mu $ satisfying \eqref{dnu}. 
\\
But having in mind that $\cG_+= \Pi\circ\cN^+$ and $\cG_-= \Pi\circ\cN^-$, it follows by a straightforward computation, that 
\begin{align}
\frac{<d\cN^+,\cN^->}{1-<\cN^-,\cN^+>} &= d\log(1 + \|\cG_+\|^2) -  \Re\left( \frac{ 2d\cG_+}{\cG_+-\cG_-}\right)\label{g1}\\
\frac{<d\cN^-,\cN^+>}{1-<\cN^-,\cN^+>} &= d\log(1 + \|\cG_-\|^2) -  \Re\left( \frac{ 2d\cG_-}{\cG_--\cG_+}\right)\label{g2}\\
1-<\cN^-,\cN^+> & = \frac{ 2 \ \| \cG_+-\cG_-\|^2}{(1 + \|\cG_+\|^2) (1 + \|\cG_-\|^2) }\label{g3}
\end{align}
and the existence of $\lambda$ and $\mu$ is equivalent to  \eqref{existence}. Moreover, using \eqref{eq5}, \eqref{eq6}, \eqref{g1},  \eqref{g2} and \eqref{g3} , $\bX$ and $\bN$ can be written as in \eqref{eq5g} and \eqref{eq6g}, where  $\rho^+$ and  $\rho^-$ are  given by \eqref{rhoplus},  and \eqref{rhominus}.
\end{proof}
\begin{pro} Let $\bX: \Sigma \longrightarrow \H^3$,
$ \bX= (r, \x)$ be a frontal  with unit normal vector  $\bN=(s,\n)$. If we denote by 
 $I$, $II$ and $III$   the first, second and third fundamental form of $\bX$ and  by $I^\pm$, $II^\pm$ and $III^\pm$ the corresponding fundamental forms of its associated frontals, $\cX^\pm$,  then the following relations hold:
\begin{align}
I & =  I^\pm+ r^2 III^\pm - 2r II^\pm,\nonumber\\
III& =  I^\pm+ s^2 III^\pm \pm 2s II^\pm,\label{eq7}\\
II& = \pm I^\pm- r s  III^\pm +(s\mp r) II^\pm.\nonumber
\end{align}
\end{pro}
\begin{proof}
Using  \eqref{a1},  \eqref{a2},  \eqref{n1} and  \eqref{n2},  we have,
\begin{align*} \cX^+ + r  \cN^+ & =  \cX^- + r  \cN^-=  \x \\
 \cX^- + s \  \cN^-& = -\cX^+ + s \  \cN^+= \n
\end{align*}
and the result follows by a straightforward computation using the definition of the corresponding fundamental form.  
\end{proof}
Moreover, from (\ref{eq7}), it is also easy to prove,
\begin{cor} At the non singular points the following relations hold:
\begin{equation}
K^\pm= \frac{1 \mp 2 H + K_e}{s^2 + 2H r s + K_e r^2}, \qquad 
H^\pm= \frac{ H(s \mp r) + r K_e \mp s}{s^2 + 2H r s + K_er^2}, \label{H1K1}
\end{equation}
where $H$, $K_e$, $H^\pm$ and $K^\pm$, are the mean curvature, the extrinsic curvature of $\bX$, the mean curvature and the Gauss curvature of $\cX^\pm$, respectively.
\end{cor}

The support function $\rho^+$ of  $\cX^+$ will play an important role in our geometric construction. We now establish some useful expressions relating $\rho^+$ with the geometry of the frontal  $\bX$.

At the points where $\cN^+$ is  an immersion, we can write
\begin{equation} \cX^+ := \hat{\nabla}^+ \rho^+ + \rho^+ \cN^+, \label{rho1}
\end{equation}
where by $\hat{\nabla}^+$ we denote the gradient with respect to the spherical metric $III^+ $.

If $z=u+ iv$ is a local conformal complex parameter for $III^+$ we  may write 
$$ III^+ = 2<\cN^+_z,\cN^+_{\bar z}> |dz|^2$$
and from (\ref{rho1}),
\begin{equation}-\frac{2 H^+}{K^+} = 2\frac{<\cX^+_z,\cN^+_{\bar z}>}{<\cN^+_z,\cN^+_{\bar z}>}= \hat{\Delta}^+ \rho^+ + 2 \rho^+,\label{hk}\end{equation}
where $\hat{\Delta}^+$ is the Laplace operator respect to $III^+$.
 \begin{teo}\label{wein}
Let $\bX:\Sigma\longrightarrow \H^3$ be a frontal  with unit normal $\bN$ and 
$$\cX^+ = \hat{\nabla}^+ \rho^+ + \rho^+ \cN^+$$ one of its associate  frontals in $\R^3$. Then, at the points where $\cN^+$ is an immersion, we have
$$ 2 \Theta (H-1) + (1-\Theta)K_I = 0,$$
where 
\begin{equation} \Theta=(\rho^+)^2+ \rho^+ \hat{\Delta}^+\rho^+ -|\hat{\nabla}^+ \rho^+|^2.\label{curvature}
\end{equation}
and $H$ and $K_I$ are the mean curvature and the Gauss curvature of $\bX$.
\end{teo}
\begin{proof} From \eqref{eta2},  \eqref{rhop}, \eqref{rhom}, \eqref{eq5} and \eqref{eq6}, we have that 
\begin{equation}
r + s = -\frac{1}{\rho^+},  \qquad  r-s = -\frac{1}{\rho^-}= -\frac{(\rho^+)^2 + |\hat{\nabla}^+ \rho^+|^2}{\rho^+}\label{r+s}
\end{equation}
From \eqref{H1K1}, \eqref{hk} and \eqref{r+s},
\begin{eqnarray*}\Theta &=& \rho^+ \hat{\Delta}^+ \rho^+ +  (\rho^+)^2 -|\hat{\nabla}^+ \rho^+|^2 = -2 \rho^+\frac{H^+}{K^+} - ( |\hat{\nabla}^+ \rho^+|^2 + (\rho^+)^2) = \\
&= & \frac{K_e-1}{1 - 2 H + K_e}.
\end{eqnarray*}
or equivalently
$$ 2 \Theta (H-1) + (1-\Theta)(K_e-1)=0,$$
which concludes the proof.
\end{proof}
\begin{remark} \label{r2}{\rm Equation \eqref{curvature} means that the metric
$$ \frac{1}{(\rho^+)^2} III^+$$
has  curvature $\lambda$.}
\end{remark}

\subsection{Ribaucour transformations}

In the previous subsection we  related a  frontal  in $\mathbb{H}^3$ with the two envelopes of a congruence of spheres in $\R^3$. Our aim is to show that if we start with a flat front in $\H^3$, then we can transform the induced congruence of spheres into another congruence of spheres having minimal surfaces as envelopes. 

With this goal in mind, we will now briefly recall some material from the classical theory of Ribaucour transformations. Such transformations are closely related to the theory of cyclic systems, developed by Bianchi in  \cite{Bi}. More information about this kind of transformations can be found in \cite{Bi, DFT, CFT, CFT2, LRTT,LT}.

\
 
 We  start with some general facts. For the proofs and more details, see \cite{Bi, CFT}.

Let $\widetilde{M}$ and $M$ be oriented surfaces in $\mathbb{R}^3$ and  denote by $\widetilde{{N}}$ and ${N}$ their  respective Gauss maps. 
\begin{defi}\label{d1} We say $\widetilde{M}$  is obtained from $M$ by a Ribaucour transformation if there is a smooth function $\tau:M \longrightarrow \R$ and a diffeomorphism ${\cal H}:M \longrightarrow \widetilde{M}$ satisfying
\begin{itemize}
\item $p + \tau(p) {N}(p) = {\cal H}(p) + \tau(p) \widetilde{{N}}({\cal H}(p))$, for all $p\in M$.
\item $\{p + \tau(p) {N}(p) : \ p\in M\}$ is a two-dimensional manifold.
\item ${\cal H}$ preserves lines of curvature.
\end{itemize}
\end{defi}

 \begin{pro}[\cite{Bi, CFT}] \label{sc} If $M$ is a simply-connected surface in $\mathbb{R}^3$, which admits orthogonal principal
direction vector fields, then  $\widetilde{M}$ is obtained   from $M$ by a Ribaucour transformation if and only if there exists a regular function $\tau$, $\tau=-\phi/\rho$  where   $\phi$  and $\rho$  are solutions of the following system of differential equations.
 \begin{align}
d\rho & = < \nabla^m \phi,d{\cal N}>, \label{rho}
\end{align}
and where by $\nabla^m$ we denote  the gradient with respect to the first fundamental form $I^m$ of $M$.

Using the functions $\rho$ and $\phi$,  we have that the following relations hold, 
\begin{align}
\widetilde{Z} & = Z - \frac{2 \phi}{\| \nabla^m \phi\|^2 + \rho^2}(  \nabla^m \phi + \rho N ), \label{ribt1}\\ 
\widetilde{N} & = N - \frac{2 \rho}{\| \nabla^m \phi\|^2 + \rho^2}(  \nabla^m \phi + \rho N ). \label{ribt2}
\end{align}
where by  $\widetilde{Z}$ and $Z$ we denote parametrizations of $\widetilde{M}$ and $M$, respectively.
\end{pro}

\

The existence of the above  functions has been essential in the study of classical Ribaucour transformations. With the aim of extending this study to either surfaces of non trivial topology or to surfaces admitting some kind of singularities, we give the following definition:
\begin{defi}{\rm 
Let $\cZ, \widetilde{\cZ}:\Sigma \longrightarrow \R^3$ be two frontals with unit normals  ${\cal N}$  and  $\widetilde{{\cal N}}$, respectively,  and such that $\cN \neq \widetilde{{\cal N}}$ everywhere.  We say that the pair $(\cZ, \widetilde{\cZ})$ is {\sl Ribaucour integrable}  if and only if 
\begin{itemize}
\item $\cZ$ and $\widetilde{\cZ}$  are envelopes to a congruence of spheres, 
\item and
\begin{equation}
\Re \left( \int_\gamma \ \frac{d\widetilde{\cG}}{\widetilde{\cG} - \cG}\right) = 0, \qquad \text{\rm for any loop $\gamma$ in $\Sigma$},  \label{RP}
\end{equation}
where $\Pi\circ\cN = \cG$ and $\Pi\circ \widetilde{\cN}=\widetilde{\cG}$.
\end{itemize}}
\end{defi}
We have the following characterization result:
\begin{teo} \label{t02}  
Let $\cZ, \widetilde{\cZ}:\Sigma \longrightarrow \R^3$ be  frontals with unit normals  ${\cal N}$  and  $\widetilde{{\cal N}}$, respectively. 
Then   $(\cZ, \widetilde{\cZ})$ is  Ribaucour integrable  if and only if  there exist a frontal $\cX_\cZ:\Sigma  \longrightarrow \R^3$ with the same unit normal $\cN$ as $\cZ$ and a regular function $\tau$, $\tau=-\phi/\rho$, where   $\phi$  and $\rho$  are smooth functions satisfying 
\begin{align}
\rho =  <\cX_\cZ,\cN>, \qquad
 d\rho  =  < \cX_\cZ, d\cN>, \qquad 
 d\phi = < \cX_\cZ, d\cZ> \label{ribeq}
\end{align}
and such  that  the following relations hold,
\begin{align}
\widetilde{\cZ} & = \cZ - \frac{2 \phi}{\| \cX_\cZ\|^2 }\cX_\cZ ,\label{ribeq1}\\
\widetilde{\cN} & = \cN - \frac{2 \rho}{\|\cX_\cZ\|^2 }\cX_\cZ.\label{ribeq2}
\end{align}
 \end{teo}
 \begin{proof}
 First, we observe from \eqref{g1} that \eqref{RP} holds if and only if there exists a well defined regular function $\lambda:\Sigma\longrightarrow \R^+$
 $$\log( \lambda) (p) = \int^p_{p_0}\frac{< d\widetilde{\cN},\cN>}{1 - <N,\widetilde{\cN}>}$$
 such that $ \cX_\cZ= \lambda(\cN - \widetilde{\cN})$ is a frontal with unit normal $\cN$.
\\ 
Thus, if $\cZ$ and $\widetilde{\cZ}$ are envelopes to a congruence of spheres of centers $$\cC:= \cZ + \tau  \cN =  \widetilde{\cZ} + \tau \widetilde{{\cal N}}$$ and radius $\tau$ and $(\cZ, \widetilde{\cZ})$ is Ribaucour integrable, 
we have the existence of $\cX_\cZ$. 
Take $\rho=<\cX_\cZ,\cN>$ the support function of $\cX_\cZ$ and consider $\phi = - \tau \rho$,  then we have that 
$$ \rho =<\cX_\cZ,\cN>, \qquad d\rho = <\cX_\cZ,d\cN>$$
and 
\begin{align*}  d\phi &= - d\tau < \cX_\cZ,\cN> - \tau < \cX_\cZ, d\cN> = -< \cX_\cZ,d(\tau\cN)>,
\\
& = -<\cX_\cZ,d\cC> + <\cX_\cZ,d\cZ> = <\cX_\cZ,d\cZ> ,
\end{align*}
which proves  \eqref{ribeq}.
\\
Moreover, $ (1 - < N, \widetilde{\cN} >)  \lambda  =\rho$ and $ 2 \rho^2 = \| \cX_\cZ\|^2 (1 - < N, \widetilde{\cN} >)$. Thus
$$   \cN - \widetilde{\cN} = \frac{1 - < N, \widetilde{\cN} >}{\rho} \cX_\cZ = \frac{2\rho}{\| \cX_\cZ\|^2} \cX_\cZ,$$ 
and
$$ \widetilde{\cZ}- \cZ= -\frac{\phi}{\rho}( \cN - \widetilde{\cN} ) =  - \frac{2 \phi}{\| \cX_\cZ\|^2 }\cX_\cZ.$$
which gives \eqref{ribeq1} and \eqref{ribeq2}.
\\
The converse is clear from  \eqref{ribeq}, \eqref{ribeq1}, \eqref{ribeq2} and the observation made at the beginning of the proof.
 \end{proof} 
 \begin{defi}{\rm  The above  $(\rho, \phi)$ are called  {\sl Ribaucour data} of the Ribaucour integrable  pair $(\cZ,\widetilde{\cZ}) $.}
\end{defi}
\begin{remark}\label{rk01}{\rm Observe that if $\cZ$ is a front and $(\rho, \phi)$  are Ribaucour data of $(\cZ, \widetilde{\cZ})$, then $\cX_\cZ$ is uniquely determined by $\rho$ and $\phi$. Moreover, in the particular case that  $\cZ$ is an immersion one also has, from \eqref{ribeq},  that 
\begin{equation} 
\label{assorib}
\cX_\cZ = \nabla^m \phi + \rho \cN,
\end{equation}
where $\nabla^m$ denotes the gradient respect to the first fundamental form of $\cZ$. }
\end{remark}
 \begin{remark}\label{rk03}{\rm
From Definition \ref{d1} and Proposition \ref{sc}, if  $\Sigma$ is simply connected and $\cZ:\Sigma \longrightarrow \R^3$ is an  immersion without umbilic points, then  $\widetilde{\cZ}:\Sigma \longrightarrow \R^3$ is obtained from $\cZ$ by a Ribaucour transformation if and only if  the pair $(\cZ, \widetilde{\cZ})$ is  Ribaucour integrable.}
 \end{remark}
\begin{remark}\label{rk1}{\rm
From   \eqref{ribeq},   \eqref{ribeq1},  \eqref{ribeq2} and \eqref{assorib}, if $(\rho,\phi)$ are   Ribaucour data of  a Ribaucour integrable pair $(\cZ,\widetilde{\cZ})$ , then for any real constant $k\neq 0$,  $(k \rho, k \phi)$ are also Ribaucour data  of the same pair $(\cZ,\widetilde{\cZ})$.} 
\end{remark}
\begin{lem}\label{lema1}
If ($\cZ$, $\widetilde{\cZ}$) is a  Ribaucour integrable pair of frontals with  Ribaucour data $(\rho,\phi)$, then ( $\widetilde{\cZ}$, $\cZ$) is also Ribaucour integrable and   Ribaucour data $(\widetilde{\rho}, \widetilde{\phi}) $  of ( $\widetilde{\cZ}$, $\cZ$) are given by
\begin{align}
&\widetilde{\rho} =  \frac{\rho}{\|\cX_\cZ\|^2 }, 
&\widetilde{\phi}  = \frac{\phi}{\| \cX_\cZ\|^2}. \label{itrans}
\end{align}
\end{lem}
\begin{proof}
If we take  $\widetilde{\rho}$ and $\widetilde{\phi} $ as in \eqref{itrans} and consider 
\begin{equation} \widetilde{\cX}_{\widetilde{\cZ}} \ = \ \frac{\widetilde{\rho}}{\rho}  \  \cX_\cZ, \label{relationassociated}\end{equation}
then from \eqref{ribeq},  \eqref{ribeq1},  \eqref{ribeq2} and by a straightforward computation, we see that the following relations hold:
\begin{align*}
\widetilde{\rho} =  <\widetilde{\cX}_{\widetilde{\cZ}},\widetilde{\cN}>, \qquad
 d\widetilde{\rho}  =  < \widetilde{\cX}_{\widetilde{\cZ}}, d\widetilde{\cN}>, \qquad 
 d\widetilde{\phi} = < \widetilde{\cX}_{\widetilde{\cZ}}, d\widetilde{\cZ}> 
\end{align*}
and 
\begin{align*}
 \cZ & =  \widetilde{\cZ}  - \frac{2\widetilde{ \phi}}{\|\widetilde{\cX}_{\widetilde{\cZ}}\|^2 }\widetilde{\cX}_{\widetilde{\cZ}}\\
 \cN & = \widetilde{\cN}  - \frac{2 \widetilde{\rho}}{\|\widetilde{\cX}_{\widetilde{\cZ}}\|^2 }\widetilde{\cX}_{\widetilde{\cZ}}
\end{align*}
which concludes the proof.
\end{proof}
\begin{pro} \label{p3}
If ($\cZ$, $\widetilde{\cZ}$) is  Ribaucour integrable  with  Ribaucour  data $(\rho,\phi)$ and $\widetilde{\rho}$ and $\widetilde{\phi} $ are as in \eqref{itrans},  then 
the frontals 
\begin{align}
& \cX^+ :=   \cX_\cZ, \qquad  \cX^- := \widetilde{\cX}_{\widetilde{\cZ}},
\label{assrib}
\end{align}
are the associated frontals of a frontal $\bX$ in $\H^3$ with hyperbolic Gauss maps $\cG = \Pi\circ \cN$ and $ \widetilde{\cG }= \Pi\circ\widetilde{\cN}$, respectively.
\end{pro}
\begin{proof}
It is clear that $\cX^+$ and $\cX^-$ are frontals in $\R^3$ with unit normal vectors $\cN$ and $ \widetilde{\cN}$ respectively.
Moreover, from \eqref{itrans} and  \eqref{relationassociated}, 
\begin{align*}
& \rho = \widetilde{\rho}\  \|\cX^+\|^2 , \qquad  \widetilde{\rho}  = \rho\ \|\cX^-\|^2,\\
 &   \cX^+ = \frac{1}{2  \widetilde{\rho}} ( \cN - \widetilde{\cN}),  \qquad  \cX^- = \frac{1}{2 \rho }( - \cN  + \widetilde{\cN}) .
\end{align*}
Thus,  the symmetry condition \eqref{sc} is satisfied and  $\cX^+$ and $\cX^-$  are envelopes of a congruence of spheres with center
$$\x = \cX^+ + r \cN = \cX^- + r\widetilde{\cN},$$
and radius $$r= -\frac{1}{2\rho} - \frac{1}{2\widetilde{\rho}} .$$ 
From Proposition \ref{p1} we conclude that the frontal  $\bX$ given by
\begin{equation*}
\bX  =  -\frac{1}{2\rho} (1,\cN ) -  \frac{1}{2\widetilde{\rho}} (1,\widetilde{\cN} ).
\end{equation*}
satisfies the required properties. \end{proof}
\begin{remark} {\rm The above proposition gives a geometric relationship between  Ribaucour integrable pairs of frontals in $\R^3$  and frontals in $\H^3$. }
\end{remark}

Restricting ourselves to minimal surfaces in $\R^3$ and having in mind Remark \ref{rk03}, we may use some results in \cite{Bi,CFT2,LT} in order to prove the following proposition:
\begin{pro}\label{fp}
Let $\cZ: \Sigma \longrightarrow \R^3$ be a non totally umbilical  minimal immersion   with Gauss map ${\cal N}$. Assume there exist two smooth functions $\rho,\phi:\Sigma\longrightarrow \R$ satisfying  
\begin{align}
d\rho &= < \nabla^m\phi, d\cN>\label{rho}\\
{\rm Hess} (\phi) & =  \rho c \  I^m + (\rho - c \phi) \ II^m \label{hessiano}\\
|\nabla^m \phi|^2 & = - \rho^2 + 2 c \rho \phi \label{modulo}
\end{align}
for some real constant $c$, $c\neq 0$, where  $\nabla^m$ denotes the gradient respect to the first fundamental form of $\cZ$. Then 
\begin{equation} 
\widetilde{\cZ} = \cZ - \frac{1}{c \rho}\left( \nabla^m \phi + \rho N^m\right), \label{rb}
\end{equation}
is a minimal immersion with Gauss map 
\begin{equation} 
\widetilde{{\cal N}} = - \frac{1}{c \phi} \left( \nabla^m \phi + \rho {\cal N}\right)  +
 {\cal N}, \label{nrb}
\end{equation}
and $(\cZ,\widetilde{\cZ} )$ is Ribaucour integrable with   Ribaucour  data $(\rho,\phi)$.
\\
Moreover,  away from umbilics, any minimal Ribaucour transformation of $\cZ$  is, locally, obtained as in \eqref{rb} for some regular functions $\rho$ and $\phi$ satisfying \eqref{rho}, \eqref{hessiano} and \eqref{modulo}.
\end{pro}
\begin{defi} {\rm 
When $ 2 c =1$, we  shall also say that the Ribaucour integrable pair $(\cZ,\widetilde{\cZ})$ is a}  minimal normalized  Ribaucour pair.
\end{defi}
\begin{remark} \label{rk2}
{\rm  If   $(\cZ, \widetilde{\cZ})$  is a minimal  Ribaucour integrable pair, then, by  applying the homothety $\cH_{2c}:\mathbb{R}^3\longrightarrow \mathbb{R}^3$ of ratio $2 c$, we get that  the pair  $\left(\cH_{2c}\circ\cZ, \cH_{2c}\circ\widetilde{\cZ}\right)$ is a minimal normalized Ribaucour pair.}
\end{remark}
From \eqref{itrans} and \eqref{modulo} we have,
\begin{pro}\label{cR} Let   $(\cZ,\widetilde{\cZ})$ be a minimal normalized Ribaucour pair and $(\rho,\phi)$ Ribaucour data of $(\cZ,\widetilde{\cZ})$. Then $(\widetilde{\cZ},\cZ)$ is also a minimal normalized Ribaucour pair and $(1/\phi, 1/\rho)$ are Ribaucour data of $(\widetilde{\cZ},\cZ)$.
\end{pro}

\section{Flat surfaces in $\H^3$ and minimal normalized Ribaucour pairs}
\label{flatminimal}
In this section we will show that if we start with a flat front in $\H^3$ and consider the induced congruence of spheres discussed in subsection \ref{spacespheres}, then there is an explicit way to transform this congruence of spheres into one that defines a Ribaucour transformation between minimal surfaces. We will start with a review of  flat  fronts in $\H^3$.
\subsection{Flat fronts in $\H^3$}
 Let $\bX:\Sigma \longrightarrow \H^3$ be a
flat front with unit normal vector $\bN$. Then, the metric  $d\sigma^2:= I+III$
inherits a
canonical Riemann surface structure such that 
$d\sigma^2$ is hermitian. This canonical Riemann surface structure
provides a conformal representation for the immersion $\bX$ that allows one to represent any flat front in $\H^3$ in terms of holomorphic data (see  \cite{GMi},
\cite{GMM}  and \cite{KUY} for the details). 

Actually,  the
hyperbolic Gauss maps $\cG_-,\cG_+:\Sigma \longrightarrow\C\cup\{\infty\} $  are
holomorphic and we can  recover flat fronts in terms of  $\cG_-$ and $\cG_+$. In fact, adapting Theorem
2.11 in \cite{KUY} to our model of hyperbolic space, we have the following holomorphic representation:
\begin{teo}[\cite{KUY}] \label{teokuy}
Let $\cG_-$ and $\cG_+$ be non-constant meromorphic functions on a Riemann surface $\Sigma$
such that $\cG_-(p)\neq \cG_+(p)$ for all $p\in \Sigma$. Assume that
\begin{enumerate}
\item all the poles of the 1-form $\frac{d\cG_-}{\cG_- - \cG_+}$ are of order 1, and
 \item $\Re\int_\gamma \frac{d\cG_-}{\cG_- - \cG_+}=0$, for each loop $\gamma$ on $\Sigma$.
\end{enumerate}
Set 
\begin{equation}
\xi_-:= c_0 \exp \int \frac{d\cG_-}{\cG_--\cG_+}, \quad \text{and}\quad 
\xi_+:= c_1 \exp \int \frac{d\cG_+}{\cG_+-\cG_-},\label{xi}
\end{equation}
where $c_0$ and $ c_1$ are non zero complex numbers such that  $\xi_{+}\xi_{-}=\cG_{+}-\cG_{-}$. 
Then,  the map $\bX=(x_0,x_1,x_2,x_3): \Sigma\longrightarrow \H^3$ given by
\begin{align*}
x_1 + i \   x_2 & = \frac{\cG_-}{|\xi_-|^2} + \frac{\cG_+}{|\xi_+|^2} \\
x_0 & = \frac{1}{2}\left(  \frac{|\cG_-|^2 + 1}{|\xi_-|^2} + \frac{|\cG_+|^2 + 1}{|\xi_+|^2}\right)\\
x_3 & = \frac{1}{2}\left(  \frac{|\cG_-|^2 - 1}{|\xi_-|^2} + \frac{|\cG_+|^2 - 1}{|\xi_+|^2}\right)
\end{align*}
is singly valued on $\Sigma$ and  $\psi$ is a flat front if and only
if $\cG_-$ and $\cG_+$ have no common branch points. Moreover its unit normal vector $\bN=(n_0,n_1,n_2,n_3)$ is given  by 
\begin{align*}
n_1 + i \   n_2 & = -\frac{\cG_-}{|\xi_-|^2} + \frac{\cG_+}{|\xi_+|^2} \\
n_0 & = \frac{1}{2}\left( - \frac{|\cG_-|^2 + 1}{|\xi_-|^2} + \frac{|\cG_+|^2 + 1}{|\xi_+|^2}\right)\\
n_3 & = \frac{1}{2}\left(  -\frac{|\cG_-|^2 - 1}{|\xi_-|^2} + \frac{|\cG_+|^2 - 1}{|\xi_+|^2}\right)
\end{align*}

Conversely, any non-totally umbilical flat front can be constructed in this
way.
\end{teo}

\subsection{The geometric link between flat fronts in $\H^3$ and minimal surfaces in $\R^3$} 
We are now ready to show how flat fronts in $\H^3$ and minimal surfaces in $\R^3$ are related in a geometric way. Our next result shows how a given flat front is related to a normalized Ribaucour pair of minimal surfaces. 
\begin{teo}\label{direct}Let $\bX:\Sigma \longrightarrow \H^3$ be a non totally umbilical
flat front, with hyperbolic Gauss maps $\cG_-$ and $\cG_+$. Consider the 1-forms given by
\begin{equation}
\omega_- := \frac{4d\cG_+}{(\cG_+-\cG_-)^2}, \qquad \omega_+:= \frac{4d\cG_-}{(\cG_--\cG_+)^2}.\label{ms}
\end{equation}
Then, $(\omega_+,\cG_+)$  and $(\omega_-,\cG_-)$  are Weierstrass data for a minimal normalized Ribaucour pair  $(\cZ^+,\cZ^-)$, maybe branched on some cover of $\Sigma$,  in  $\R^3$.
\end{teo}
\begin{proof}
Consider $\cZ^-$ and $\cZ^+$ minimal immersion, maybe branched on some cover of $\Sigma$,  in $\R^3$ with  Weierstrass data $(\omega_-,\cG_-)$ and $(\omega_+,\cG_+)$ and with Gauss map $\cN^-$ and $\cN^+$ respectively.

If $\cX^+$ and $\cX^-$ are the associated fronts of $\bX$, then from \eqref{eq5},  \eqref{eq6}  and Theorem \ref{teokuy},  their  support functions $\rho^+$ and $\rho^-$ are given by
\begin{equation}
\rho^+ = -\frac{|\xi_+|^2}{1 + |\cG_+|^2}, \qquad \rho^-=  -\frac{|\xi_-|^2}{1 + |\cG_-|^2}.  \label{rhph}
\end{equation}
We want to prove that  the smooth functions $\rho=\rho^+$ and $ \phi=1/\rho^-$ satisfy  equation  \eqref{rho}. To see this, we can argue  away from umbilic points of  $\bX$.  In fact,  around any non umbilic point,  we can take a complex parametrization  $\bX:\Sigma \longrightarrow \H^3$, $\bX=\bX(z)$ so that 
\begin{equation} 4 d\cG_+ d\cG_- = (\cG_+-\cG_-)^2 dz^2, \label{newp}
\end{equation}
see for instance \cite{MST}. Using  this parameter,  the first and second fundamental form of $\cZ^+$ and $\cZ^-$ can be written as
\begin{align}
& I_+^m  = \frac{(1 + |\cG_+|^2)^2}{4|\cG_+'|^2} |dz|^2,  & II^m_+ = -\frac{1}{2}\left( dz^2 + d\bar{z}^2\right) .\label{fsm2}\\
& I_-^m  = \frac{(1 + |\cG_-|^2)^2}{4|\cG_-'|^2} |dz|^2,  & II^m_- = -\frac{1}{2}\left( dz^2 + d\bar{z}^2\right) ,\label{fsm}
\end{align}
see \cite{O}.  

Using \eqref{xi},  \eqref{newp},  \eqref{rhph} we obtain 
\begin{align}
 \rho^+_z  (1 + |\cG_+|^2)^2 & =   4 \left(\frac{1}{\rho^-}\right)_{\bar{z}} |\cG_+'|^2 .\label{rhz}
\end{align}
From     \eqref{fsm2}  and   \eqref{rhz} we may conclude that $\rho=\rho^+$ and $ \phi=1/\rho^-$ satisfy  equation  \eqref{rho} and consequently,
$$ \cX^+ = \nabla^m_+ \phi + \rho \cN^+,$$
where $\nabla^m_+ $ denotes de Levi-Civita connection of $M^+$. 
From  \eqref{eta2},  \eqref{rhop} and   \eqref{rhom},  
$ \|\nabla^m_+\phi\|^2 + \rho^2 = \rho \phi $ and from Proposition \ref{fp},  we conclude that    if  $\widetilde{Z}$ is given by 
\begin{equation} \widetilde{\cZ} = \cZ^+ - \frac{2}{\rho}(\nabla^m_+\phi + \rho N^+),\label{mnrp}\end{equation}
the pair $(\cZ^+,\widetilde{\cZ})$ is a minimal normalized Ribaucour pair and the normal of $\widetilde{\cZ}$ is given by
\begin{equation} \widetilde{\cN} =  \cN^+ - \frac{2}{\phi}(\nabla^m_+\phi + \rho N^+).\label{nmrp}\end{equation}
Thus, from  \eqref{rho}, \eqref{hessiano}, \eqref{mnrp} and \eqref{nmrp} we get
\begin{align*}
d\widetilde{\cZ} & = -\frac{\phi}{\rho} d\cN + 2 \frac{d \rho}{\rho^2}  (\nabla^m_+\phi + \rho\cN^+),\\
d\widetilde{\cN} & = -\frac{\rho}{\phi} d\cZ^+ + 2 \frac{d \phi}{\phi^2}  (\nabla^m_+\phi + \rho N^+). 
\end{align*}
and using \eqref{rho}, \eqref{ms}, \eqref{fsm2} and \eqref{rhph} we obtain that the first and second fundamental form of $\widetilde{\cZ}$, $\widetilde{I}$ and $\widetilde{II}$, respectively, satisfy
\begin{align*}
\widetilde{I} & = \frac{\phi^2}{\rho^2} III^+ = I^m_-,\\
\widetilde{II} & =II^m_- .
\end{align*}
Therefore, the minimal surface given by $\widetilde{\cZ}$ admits $(\omega_-,\cG_-)$ as Weierstrass data and this concludes the proof.
\end{proof}

Our next result shows that conversely, by using the geometric construction described by Proposition \ref{p3}, we can recover any non totally umbilic  flat front  from a minimal normalized Ribaucour pair.

\begin{teo} \label{converse}Let  $(\cZ,\widetilde{\cZ}):\Sigma\longrightarrow\R^3$ be a minimal normalized Ribaucour pair   in $\R^3$ with respective Gauss maps, $\cN$ and $\widetilde{\cN}$.  If  $(\rho, \phi)$ are Ribaucour data of $(\cZ,\widetilde{\cZ})$ , then 
\begin{equation}
\cX^+ := \nabla^m_+ \phi + \rho \cN, \qquad \cX^- := -\frac{\nabla^m_-\rho}{\rho^2} + \frac{1}{\phi} \widetilde{	\cN}
\end{equation}
are the associated fronts  of a flat front
$\bX:\Sigma \longrightarrow \H^3$ with hyperbolic Gauss maps $\cG_+=\Pi\circ \cN$ and $\cG_-=\Pi\circ \widetilde{\cN}$, respectively,  and where 
$\Pi$, $\nabla^m_+$ and $\nabla^m_-$ denote the usual stereographic projection and  the gradient operators with respect to the first fundamental forms of $\cZ$ and $\widetilde{\cZ}$ respectively.

Moreover, if $\omega_+$ and $\omega_-$ are as in \eqref{ms}, then  $(\omega_+,\cG_+)$ and $(\omega_-,\cG_-)$ are Weierstrass data of the  minimal normalized Ribaucour pair  $(\cZ,\widetilde{\cZ})$.
\end{teo}
\begin{proof}
From Proposition \ref{p1} the frontal  $\bX$ given by
\begin{equation*}
\bX  =  -\frac{1}{2\rho} (1,\cN ) -  \frac{\phi}{2} (1,\widetilde{\cN} ).
\end{equation*}
has $\cX^+$ and $\cX^-$ as associated  frontals.

Moreover, from  Proposition \ref{fp}, 
$$ D_X\nabla \phi = \frac{\rho}{2}  X  + (\frac{\phi}{2} - \rho) \ d\cN(X),  $$
for any tangent vector field $X$ along $\cZ$, where $D$ denotes the Levi-Civita connection of the first fundamental form of $\cZ$.
\\
Thus $$d\cX^+ =  \frac{\rho}{2}  d\cZ + \frac{\phi}{2}  \ d\cN$$
and $\cX^+$ is a front.
\\
Using also   \eqref{hessiano}, we get  $ \hat{\Delta}^+ \rho = -2 \rho + \phi$ and then we have  that in \eqref{curvature},  $\Theta=0$ which proves the map  $\bX$  is a flat front in $\H^3$.
\end{proof}
\begin{cor}
\label{c4}Let $\cZ:\Sigma\longrightarrow \R^3$ be a non totally umbilical  minimal immersion  with Weierstrass data $(\omega,g)$ and  $h$  be a solution 
of the following ordinary differential equation
\begin{equation} \label{re}
d h = k h^2 \omega - d g, \qquad \text{on $\Sigma$}
\end{equation}
such  that 
$$ \Re\int_\gamma \frac{dg}{h} = 0,  \qquad \text{ for any loop $\gamma$ on $\Sigma$},$$
where $k$ is a non-zero real constant. If $\widetilde{\Sigma}=\{ p\in \Sigma \ | \ h(p)\neq 0\}$, then there exists $\widetilde{\cZ}:\widetilde{\Sigma}\longrightarrow \R^3$ a well-defined minimal immersion with Weierstrass data 
\begin{equation}
\left(\frac{1}{ k h^2} d g, g + h\right),\label{nwd}
\end{equation} such that  $(\cZ, \widetilde{\cZ})$ is  a Ribaucour integrable pair.\\
Conversely, any minimal Ribaucour integrable pair $(\cZ, \widetilde{\cZ})$   can be obtained in this way. 
\end{cor}
\subsection{ New and old ends}

From now on $\D_\epsilon$ will denote the open disk of radius $\epsilon$ centered at the origin and  $\D_\epsilon^\star= \D_\epsilon\setminus\{0\}$.  
\begin{teo}\label{ft1} Let $\cZ:\D_\epsilon \longrightarrow \R^3$ be a conformal parametrization of a minimal surface without umbilical points. If $\widetilde{\cZ} :\D_\epsilon^\star \longrightarrow \R^3$ is a minimal end which is obtained from $\cZ$ by a Ribaucour transformation, then $\widetilde{\cZ}$ is a complete planar embedded end. 
\end{teo}
\begin{proof}
From Remark \ref{rk03} we have that $(\cZ, \widetilde{\cZ})$ is Ribaucour integrable and without loss of generality, we may assume that $g(z)$, the Gauss map of $Z$, is such that $g(0)=0$. Also, since $\cZ$ has no umbilics, we have $g'(0)\neq 0$. The Weierstrass data $(\widetilde{\omega},\widetilde{g})$ associated to $\widetilde{\cZ}$ is given by the expressions in (\ref{nwd}). The idea of the proof is that the order of the zero of $h$ at $z=0$ controls the geometry of the end.

From (\ref{re}) it follows that $h'(0)=-g'(0)\neq 0$, so $h$ has a zero of order 1 at $z=0$. Therefore, $\widetilde{\omega}$ has a zero of order 2.

Now we recall the well known expressions of the metric $d\widetilde{s}^2$ and Gaussian curvature $\widetilde{K}$ of a minimal surface in terms of the Weierstrass data $(\widetilde{\omega},\widetilde{g})$, see \cite{O}. The following expressions hold.
\begin{equation}
d\widetilde{s}^2=\frac{1}{4}(1+|\widetilde{g}|^2)^2|\widetilde{\omega}|^2.
\label{metricweierstrass}
\end{equation}

\begin{equation}
\widetilde{K}=-\frac{16}{(1+|\widetilde{g}|^2)^4}\left|\frac{d\widetilde{g}}{\widetilde{\omega}}\right|^2.
\label{curvatureweierstrass}
\end{equation}

Substitution of (\ref{nwd}) into (\ref{metricweierstrass}) yields the following expression,

\begin{equation}
d\widetilde{s}^2=\frac{1}{4}(1+|g+h|^2)^2\left|\frac{g'}{kh^2}\right|^2\left|dz\right|^2.
\end{equation}

Now, since $h$ has a zero of order 1 at $z=0$, in the neighborhhood of $z=0$ we have the estimate below.
\[
d\widetilde{s}^2 \geq \frac{C}{|z|^4}|dz|^2,
\]
and it follows that the end is complete.
 
We also note that $\widetilde{\cZ}$ has finite total curvature. This follows directly after substitution of (\ref{nwd}) into (\ref{curvatureweierstrass}).

To prove that the end is embedded, we use the criteria for embeddedness as in \cite{LM}. In other words, we have to show that the maximum order of de poles at $z=0$ of $\Phi_{j}$, $j=1,2,3.$, defined below is exactly 2.
\[
\Phi_{1}=\frac{1}{2}(1-\widetilde{g}^2)\widetilde{\omega}, \,\, \Phi_{2}=\frac{i}{2}(1+\widetilde{g}^2)\widetilde{\omega}, \,\, \Phi_{3}=\widetilde{\omega}\widetilde{g}. 
\]

Recall that $h$ has a zero of order 1 at $z=0$, so, from (\ref{nwd}), it follows that $\widetilde{\omega}$ has a pole of order 2 at $z=0$ and therefore the maximum order of the poles of the forms $\Phi_{j}$ is indeed 2.

Finally, to prove that we have a planar end, we note that $\widetilde{g}$ has a zero of order 2 at $z=0$. This can be seen by taking successive derivatives and using (\ref{re}), and taking into account the fact that $h$ has a zero of order 1 at $z=0$. Thus, the third coordinate function of $\widetilde{\cZ}$, given by the real part of the integral of $\Phi_{3}$ has a finite limit at $z=0$ and the end is a planar end. 
\end{proof}
\begin{remark}{\rm Using a different approach, the above Theorem also  was proved in \cite{CFT2}.}
\end{remark}
\begin{teo}\label{ft2} Let $\cZ:\D_\epsilon^\star \longrightarrow \R^3$ be a conformal parametrization of a minimal surface without umbilical points. If the origin is an umbilic point of $\cZ$ and $\widetilde{\cZ} :\D_\epsilon^\star \longrightarrow \R^3$ is a minimal end such that $(\cZ,\widetilde{\cZ})$ is  Ribaucour integrable, then $\widetilde{\cZ}$ is a complete planar non-embedded end. 
\end{teo}
\begin{proof}
From the fact that $\cZ(0)$ is an umbilic point, it follows that $g$ has a zero of order $m>1$ at $z=0$. Using (\ref{re}) and the fact that $h(0)=0$ we may conclude that $h$ also has a zero of order $m$ at $z=0$. The completeness and finite total curvature of $\widetilde{\cZ}$ are proved as in the proof of theorem \ref{ft1}.
The non-embeddedness comes from the fact that by looking at (\ref{nwd}) we see that $\widetilde{\omega}$ would have a pole of order $1+m>2$ and by the criteria in \cite{LM} the end $\widetilde{\cZ}$ is not embedded.  
\end{proof}
\begin{teo}\label{ft3} Let $\cZ:\D_\epsilon^\star \longrightarrow \R^3$ be a conformal parametrization of an embedded  minimal end of catenoid type without umbilic points. If  $\widetilde{\cZ} :\D_\epsilon^\star \longrightarrow \R^3$ is a minimal surface  and  $(\cZ,\widetilde{\cZ}) $  Ribaucour integrable, then  $\widetilde{\cZ}$ is a an embedded minimal end of catenoid type with the same limiting value Gauss map as $\cZ$.
\end{teo}
\begin{proof}
Let $(\omega,g)$ and $(\widetilde{\omega},\widetilde{g})$ be, respectively, the Weierstrass data of $\cZ$ and $\widetilde{\cZ}$. Since $\cZ$ is asymptotic to a catenoid end, we may assume that its Weierstrass data defined on $\D_\epsilon^\star$ have the following form:
\[
g(z)=\eta_{1}(z)z, \,\, \omega(z)=\frac{\eta_2(z)}{z^2},
\]
where $\eta_{i}$, $i=1,2$, are holomorphic functions that extend regularly to $z=0$ with $\eta_{i}(0)=1$.

Since the Weierstrass data of $\cZ$ and $\widetilde{\cZ}$ are related by (\ref{nwd}), to control the geometry of $\widetilde{\cZ}$ we must understand the local behaviour of $h$ near $z=0$. We will show that $h$ has a zero at $z=0$ and that we the following expression is valid 
\[
-\frac{h'}{h^2}=(\frac{1}{h})'=\psi(z)/z^2,
\]
where $\psi(z)$ is holomorphic with $\psi(0)$ a non-zero real number.

The expressions above imply that $\widetilde{\cZ}$ generated by $(\widetilde{\omega},\widetilde{g})$ is indeed asymptotic to a catenoid end with the same limiting value of the Gauss map.
 
For computations, it turns out that it is simpler to consider the function $\mu(z)=\frac{1}{h(z)}$

To control $h$ and $(\frac{1}{h})'$, we will use classical results for the Riccati equation and its relation with second order linear O.D.E. as exposed in \cite{H}. 

It is a simple matter to verify that if $h$ satisfies (\ref{re}), then $\mu=\frac{1}{h}$ satisfies the Riccati equation given by
\begin{equation}\label{mre}
\mu'=g'\mu^2-kf.
\end{equation}

From the classical theory of complex O.D.E., see \cite{H} pages 113-114 and Theorem 5.3.1 (page 155), it follows that $z=0$ is a regular singular point of the second order O.D.E. associated to (\ref{mre}) that is given by
\begin{equation}
\label{ode}
w''+Pw'+Qw=0,
\end{equation}
where $P=-\frac{g''}{g'}$ and $Q=-kg'f$. 

The indicial equation for (\ref{ode}) is
\[
\nu^2+(P_{-1}-1)\nu+Q_{-2}=0,
\]
where
\[
P_{-1}=\lim_{z\rightarrow 0}zP, \,\,Q_{-2}=\lim_{z\rightarrow 0}z^2Q.
\]
From the form of $g$ and $f$, it follows after a simple computation that $P_{-1}=0$ and $Q_{-2}=-k$, so the roots of the indicial equation are 
\[
\lambda_{\pm}=\frac{1 \pm \sqrt{1+4k}}{2},
\]
and a basis of solutions (possibly multivalued) of (\ref{ode}) is given by
\[
w_{1}(z)=\zeta_{1}(z)z^{\lambda_{+}}, \,\,, w_{2}(z)=\zeta_{2}(z)z^{\lambda_{-}}+Cw_{1}(z)\ln{z},
\]
where $\zeta_{i}(z)$, $i=1,2,$ are holomorphic and non-zero at $z=0$ and $C$  is an arbitrary complex number. Note that as $h$ must be single valued, the functions $w_{1}$ and $w_{2}$ must also be single valued. This implies that $C=0$ and that $\lambda_{\pm}$ must be real and integer.

Thus the solution $w$ of (\ref{ode}) must have the form 
\[
w=c_{1}\zeta_{1}z^{\lambda_{+}}+c_{2}\zeta_{2}z^{\lambda_{-}}.
\]

Since $\lambda_{\pm}$ are integers, it follows that $\lambda_{+}-\lambda_{-}=\sqrt{1+4k}>1$, and a straightforward computation allows us to conclude that the leading term in a Laurent expansion around $z=0$ for the function $\frac{w'}{w}$ is $\frac{\lambda_{-}}{z}$ if $c_{1}\neq 0$ or $\frac{\lambda_{+}}{z}$ if $c_{1}=0$.

Thus, the leading term in a Laurent expansion around $z=0$ of $\mu=-\frac{w'}{wg'}$ is $-\frac{\lambda_{\pm}}{z}$. From the relation between $h$ and $\mu$ we conclude that $h$ has a pole of order one at $z=0$.

 If we write $\widetilde{\omega}=\widetilde{f}dz$, then, from (\ref{re}) and (\ref{nwd}), it follows that $\widetilde{f}=f-\frac{h'}{kh^2}=f+\frac{\mu'}{k}$. From the Laurent expansion for $\mu$ we deduce that the leading term in the Laurent expansion for $\widetilde{f}$ at $z=0$ is $(1+\frac{\lambda_{\pm}}{k})\frac{1}{z^2}$. Note that from the form of $\lambda_{\pm}$ we are sure that $1+\frac{\lambda_{\pm}}{k}\neq 0$.  

The above estimates for $\widetilde{\omega}$ and $\widetilde{g}$ imply that $\widetilde{\cZ}$ is asymptotic to a catenoid and has the same limiting value of the Gauss map as $\cZ$.       
\end{proof}

We will now discuss what happens with a planar minimal end (embedded or not) under a Ribaucour transformation. 
\begin{teo}
Let $\cZ:\D_\epsilon^\star \longrightarrow \R^3$ be a conformal parametrization of a complete   planar minimal  end without umbilic points. If  $\widetilde{\cZ} :\D_\epsilon^\star \longrightarrow \R^3$ is a minimal surface  and $(\cZ, \widetilde{\cZ})$ Ribaucour integrable,  then either $\widetilde{\cZ}$ is a complete  planar minimal  end or $\widetilde{\cZ}$ extends regularly to the origin.
\end{teo}
\begin{proof}
Without loss of generality, we may assume that the Weierstrass data associated to $\cZ$ is given by $f(z)=\frac{1}{z^{2+q}}$, $g'(z)=\psi(z)z^{p}$ and $g(0)=0$, where $q\in \mathbb{\cZ},q\geq 0$, $p\geq q+1$ and $\psi$ is holomorphic such that $\psi(0)\neq 0$.

We first note that the solution $h$ of (\ref{re}) extends to $z=0$ with $h(0)=0$. Otherwise, $h$ would have a pole of order at least $l\geq 0$ at $z=0$. But then on the right hand side of (\ref{re}), we would have a function with pole of order at least $l+1$ at $z=0$ while on the left hand side of (\ref{re}) a function with a pole of order $2l+2+q$ which is a contradiction.

Therefore, we may assume that $h(z)=z^{m+1}H(z)$, where $m\geq 0$ and $H(z)$ is holomorphic and such that $H(0)\neq 0$.

Using (\ref{re}) we obtain relations between $p$, $q$ and $m$. In fact, substitution of the above expressions for $f(z)$, $g(z)$ and $h(z)$ into (\ref{re}) yields
\begin{equation}
\label{ce}
(m+1)H(z)+zH'(z)=kz^{m-q}(H(z))^{2}-z^{p-m}\psi(z).
\end{equation}

We separate our analysis in several cases. 

Note first that if $m\geq q$ then $p\geq m$ due to the fact that the right hand side of (\ref{ce}) is a holomorphic function. So in this case $p\geq m \geq q$.

By the same reason we can conclude that $m<q$ implies $p<m$, which would imply $p<q$, contrary to our hypothesis. 

If we had $p>m>q$ then from (\ref{re}) we would have $H(0)=0$. Thus, we are left with two cases: either $p=m$ or $m=q$.

We first consider the case where $p=m$. From (\ref{ce}) we conclude that $H(0)=-\frac{\psi(0)}{m+1}$. The Weierstrass data of $\widetilde{\cZ}$ in this case are $$\widetilde{f}=\frac{g'}{kh^2}=\frac{\psi}{kz^{m+2}H^2}, \,\, \widetilde{g}=g+h.$$
Thus, $\widetilde{g}$ would have a zero of order $p+1$ and $\widetilde{f}$ a pole of order $m+2$. Since $m\geq 1$ the end $\widetilde{Z}$ is not embedded.

Now we will see what happens if $m=q$. In this case, using (\ref{ce}), we would have $p\geq m+1$ and $$\widetilde{f}=\frac{z^{p}\psi}{z^{2m+2}H^2}.$$

If $p>2m+1$, then $\widetilde{f}$ extends to $z=0$ as a holomorphic function and we don't have an end. If $p<2m+1$ then, depending on the order of the pole of $\widetilde{f}$ at $z=0$ the end $\widetilde{\cZ}$ can be either embedded or not. To finish our proof we will prove that the case $p=2m+1$ cannot occur.

By contradiction, assume that $p=2m+1$. We will see that this assumption implies that $\psi(0)=0$ which is impossible. For $p=2m+1$ the equation  (\ref{ce}) is written as
\begin{equation}
\label{ceo}
(m+1)H+zH'=kH^{2}-z^{m+1}\psi.
\end{equation}
   
For $z=0$ we have $H(0)=\frac{m+1}{k}$. Differentiation yields

\[
(m+2)H'+zH''=2kHH'-(m+1)z^{m}\psi-z^{m+1}\psi',
\]
and for $z=0$ we obtain either
\[
2H'(0)=2H'(0)-\psi(0),
\]
if $m=0$, or
\[
(m+2)H'(0)=2(m+1)H'(0),
\]
if $m>0$. The first case ($m=0$) cannot happen because $\psi(0)\neq 0$, and so we conclude that $H'(0)=0$.

Now consider a positive integer $r<m+1$. The derivative of order $r$ of (\ref{ceo}) can be written as follows.

\begin{multline}
(m+1)H^{(r)}+\sum_{s=0}^{r}(z)^{(r-s)}(H^{(1)})^{(s)}\binom{r}{s} \\
=k\sum_{s=0}^{r}H^{(s)}H^{(r-s)}\binom{r}{s}-\sum_{s=0}^{r}(z^{m+1})^{(r-s)}\psi^{(s)}\binom{r}{s}.
\end{multline}

At $z=0$ we have
\[
(m+1)H^{(r)}(0)+rH^{(r)}(0)=2kH^{(r)}H(0)+T,
\]
where $T$ represents a sum of products of derivatives of $H$ up to order $r-1$. Using the value of $H(0)$ we obtain
\[
(r-(m+1))H^{(r)}(0)=T(0).
\]
So, as $H^{(1)}(0)=0$ it follows that $H^{(2)}(0)=0$. By iteration we may conclude that $H^{(r)}(0)=0$ for all $1\leq r<m+1$.

Finally, the derivative of order $m+1$ of (\ref{ceo}) at $z=0$ is given by
\[
2(m+1)H^{(m+1)}(0)=2kH^{(m+1)}(0)H(0)-\psi(0).
\]
But, using the fact that  $H(0)=\frac{m+1}{k}$, this would imply that $\psi(0)=0$, a contradiction.
\end{proof}

\section{Examples}
\label{examples}
In this section we obtain  the  Ribacour transformations of the catenoid  and show how to obtain Ribacour transformations of the trinoid.  
\subsection{Ribaucour transformations of the catenoid.}
It is well known, see \cite{We}, that 
 the Weierstrass data $f(z)= z^{-2}$, $g(z)=z$, correspond  to the  catenoid. The equation (\ref{re}) for this data is
\begin{equation}
\label{catenoideq}
h'=\frac{kh^2}{z^2}-1,  
\end{equation}
and its general solution can be written as
\begin{equation}
h(z)= \frac{2z \left( C -{z}^{m} \right) }{(m+1) z^m + (m-1) C}, \label{bonnet}
\end{equation}
where $C$ is complex constant and $m = \sqrt{1 + 4 k}$.

When $C=0$, the corresponding Ribaucour transformed is again a Catenoid and its associated flat front is a helicoidal flat front, see \cite{MST}, with
$$\cG_- = \frac{m-1}{m+1} \cG_+.$$
If $C\neq 0$, the Ribaucour transform is well defined if and only if $m$ is an integer and  the corresponding Ribaucour transformed has $|m|$ embedded planar ends at the zeros of $z^m - C$ and two ends of catenoid type at $z=0$ and $z=\infty$.  In this case the associated flat front in $\H^3$ has $|m|$ ends of horospherical type and two ends of rotational type.

In Figure \ref{mirror} and using the conformal ball model for $\mathbb{H}^3$, we have the flat front corresponding  to the particular solution 
 $$h(z)=\frac{z \left( 2-{z}^{3} \right) }{2({z}^{3}+1)}.$$ 
  \begin{figure}[H]
  \begin{center}
    \includegraphics [width=0.5\linewidth]{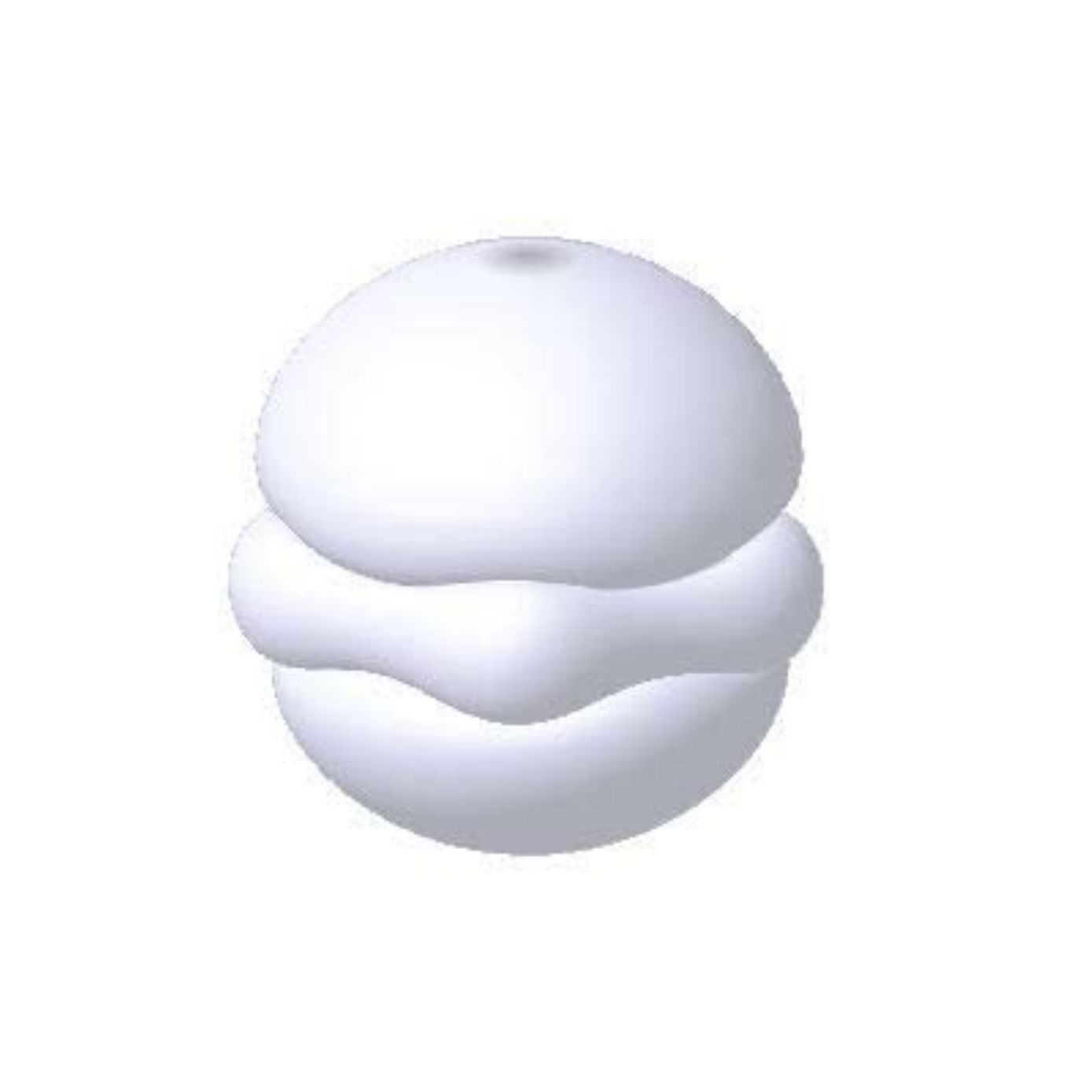}
  \end{center}
  \caption{Flat front associated to a catenoid}
  \label{mirror}
\end{figure} 

\subsection{Ribaucour transformation of the trinoid.}

A minimal trinoid is generated by the Weierstrass data given by $g(z)=z^2$ and $f(z)=\frac{1}{(z^3-1)^2}$, see \cite{JM}. The equation (\ref{re}) for this data is
\begin{equation}
\label{trinoideq}
h'=\frac{kh^2}{(z^3-1)^2}-2z.  
\end{equation}

The general solution of this equation rather complicated and involves hypergeometric functions. However, by choosing special values of the constant $k$ we get relatively simple solutions. For instance, for $k=5$, it is easy to check that
\[
h(z)=z^2(z^3-1),
\]
is a solution of (\ref{trinoideq}).

A rough sketch of the corresponding minimal surface is shown in Figure \ref{trinoid}, note that in addition to the three catenoid type ends the surface also has one planar end.
 \begin{figure}[H]
\begin{center}
\includegraphics[width=0.6\linewidth]{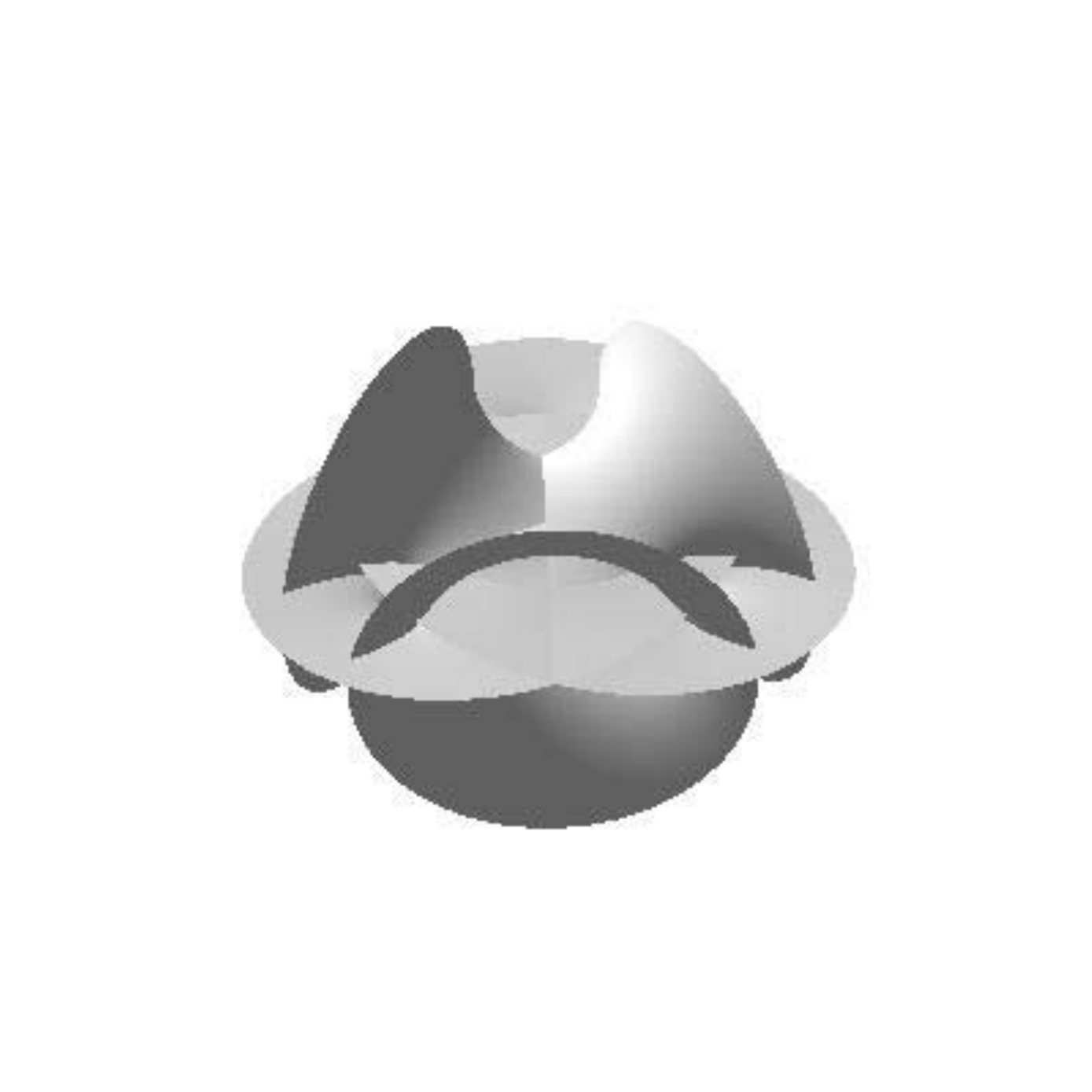} 
\end{center}
  \caption{Ribaucour transform of the trinoid}\label{trinoid}
\end{figure}

\section{Concluding remarks}
\label{finalremarks}
\begin{itemize}
{\rm 
\item Propositions \ref{p1} and \ref{p3} are the source of all the results in this paper and doubtless they will have other applications. For example, it is known that linear Weingarten surfaces in $\H^3$ of Bryant type, see \cite{GMM2}, admit a representation  by holomorphic data. To understand why this is possible, we can consider a Ribaucour integrable pair $(\cZ,\widetilde{\cZ})$  where $\cZ:\Sigma \longrightarrow \R^3$ is a minimal surface. In general, $\widetilde{\cZ}$ is not necessarily a minimal surface, but, possibly after a homothety  in $\R^3$, we have Ribaucour data $(\rho,\phi)$  so that 
$$ | \nabla^m\phi|^2 + \rho^2 = \rho \phi + \epsilon.$$  
Thus, from Theorem \ref{wein},  the front in $\H^3$ associated, via Proposition \ref{p3}, with the pair  $(\cZ,\widetilde{\cZ})$ is  a linear Weingarten front of Bryant type satisfying,
$$ 2 \epsilon (H-1) + (1-\epsilon) K_I =0.$$
This geometric relationship might be used to discover new results concerning Ribaucour transformations of minimal surfaces by using the knowledge about linear Weingarten fronts of Bryant type and vice-versa.
\item It should be also interesting to study the family of fronts in $\H^3$ coming from Ribaucour integrable pairs of surfaces with constant mean curvature in $\R^3$.
\item From Theorem  \ref{t0}, any front  $\bX: \Sigma\longrightarrow \H^3$ can be represented in terms of its hyperbolic Gauss maps $\cG_+, \ \cG_-:\Sigma \longrightarrow \S^2$
by the expression given in \eqref{eq5g}, where $\rho^+$ and $\rho^-$ are determined by \eqref{rhop}, \eqref{rhoplus} and \eqref{rhominus}. 

Observe that there exists a front in $\H^3$ with hyperbolic Gauss maps $\cN_+, \ \cN_-:\Sigma \longrightarrow \S^2$ if and only if the 1-form 
$$ \frac{<\cN^-, d\cN^+> - <\cN^+, d\cN^-> }{ 2 ( 1 - < \cN^+,\cN^->)}$$
is exact.
\item Corollary \ref{c4} reformulates the classical theory of Ribaucour transformations between minimal surfaces in terms of a complex Ricatti ordinary differential equation. This new approach simplifies and extends the classical theory to surfaces with non trivial topology.
}
\end{itemize}

\end{document}